\documentclass[a4paper,10pt]{amsart}

\usepackage{amssymb}
\usepackage{amsmath}
\usepackage{amsfonts}

\renewcommand{\appendix}[1]{
    \addtocounter{section}{1}
    \setcounter{equation}{0}
    \renewcommand{\thesection}{\Alph{section}}
    \section*{Appendix \thesection\protect\indent #1}
    \addcontentsline{toc}{section}{Appendix \thesection\ \ \ #1}
}
\newcommand\encadremath[1]{\vbox{\hrule\hbox{\vrule\kern8pt
\vbox{\kern8pt \hbox{$\displaystyle #1$}\kern8pt}
\kern8pt\vrule}\hrule}}
\def\enca#1{\vbox{\hrule\hbox{
\vrule\kern8pt\vbox{\kern8pt \hbox{$\displaystyle #1$}
\kern8pt} \kern8pt\vrule}\hrule}}

\newcommand\figureframex[3]{
\begin{figure}[bth]
\hrule\hbox{\vrule\kern8pt
\vbox{\kern8pt \vbox{
\begin{center}
{\mbox{\epsfxsize=#1.truecm\epsfbox{#2}}}
\end{center}
\caption{#3}
}\kern8pt}
\kern8pt\vrule}\hrule
\end{figure}
}
\newcommand\figureframey[3]{
\begin{figure}[bth]
\hrule\hbox{\vrule\kern8pt
\vbox{\kern8pt \vbox{
\begin{center}
{\mbox{\epsfysize=#1.truecm\epsfbox{#2}}}
\end{center}
\caption{#3}
}\kern8pt}
\kern8pt\vrule}\hrule
\end{figure}
}

\newcommand{\beq}{\begin{equation}}
\newcommand{\eeq}{\end{equation}}
\newcommand{\bea}{\begin{eqnarray}}
\newcommand{\eea}{\end{eqnarray}}

\newcommand{\Res}{\mathop{\,\rm Res\,}}

\renewcommand{\l}{\lambda}

\renewcommand{\d}{{{\partial}}}
\newcommand{\D}{{{\hbox{d}}}}

\newcommand{\Pint}{{\int\kern -1.em -\kern-.25em}}

\renewcommand{\a}{o}
\renewcommand{\l}{\lambda}
\renewcommand{\L}{\Lambda}

\renewcommand{\l}{\lambda}
\renewcommand{\L}{\Lambda}

\renewcommand{\thesection}{\arabic{section}}

\makeatletter
\@addtoreset{equation}{section}
\makeatother

\newcommand{\bn}{\mathfrak{b}}
\newcommand{\lx}{\tilde{x}}
\newcommand{\ly}{\tilde{y}}
\newcommand{\Do}{\mathcal{D}}
\newcommand{\si}{\sigma}
\newcommand{\Si}{\tilde{\sigma}}
\newcommand{\xv}{\rho}
\newcommand{\oM}{\overline{\mathcal{M}}}

\newcommand{\cH}{H^{\circ}}
\newcommand{\ch}{h^{\circ}}

\newcommand{\dih}{h}
\newcommand{\A}{\mathcal{A}}
\newcommand{\Z}{\mathbb{Z}}
\newcommand{\cE}{\mathcal{E}}
\newcommand{\br}[1]{\left( #1 \right) }
\newcommand{\bs}[1]{\left\{ #1 \right\} }
\newcommand{\ccor}[1]{\left\langle #1 \right\rangle^\circ }
\newcommand{\cor}[1]{\left\langle #1 \right\rangle }

\newcommand{\cV}{\mathcal{V}}
\newcommand{\of}{\bar f}

\newtheorem{theorem}{Theorem}[section]
\newtheorem{proposition}[theorem]{Proposition}

\newtheorem{lemma}[theorem]{Lemma}

\theoremstyle{definition}
\newtheorem{definition}[theorem]{Definition}
\newtheorem{notation}[theorem]{Notation}

\theoremstyle{remark}
\newtheorem{remark}[theorem]{Remark}

\newcommand{\1}{1\hskip-4pt1}
\newcommand{\C}{\mathbb C}
\newcommand{\End}{\mathop{\rm End}}
\newcommand{\rmtop}{\mathrm{top}}
\renewcommand{\L}{\Lambda}

\newcommand{\reg}{\rm reg}
\newcommand{\pd}[2]{\dfrac{\partial #1}{\partial #2}}
\renewcommand{\b}{\beta}
\renewcommand{\a}{\alpha}
\newcommand{\res}{\mathop{\rm res}}
\renewcommand{\l}{\lambda}
\newcommand{\s}{\sigma}
\newcommand{\cA}{\mathcal A}
\newcommand{\QED}{\ifmmode\eqno\square\else}
\newcommand{\odd}[1]{\left\lfloor #1 \right\rfloor^{-}_1}

\newcommand{\tW}{\widetilde W}

\title[Polynomiality of Hurwitz numbers and the ELSV formula]{Polynomiality of Hurwitz numbers, Bouchard-Mari\~no conjecture, and a new proof of the ELSV formula}

\author{P.~Dunin-Barkowski}

\author{M.~Kazarian}

\author{N.~Orantin}

\author{S.~Shadrin}

\author{L.~Spitz}

\address{P.~D.-B.: Korteweg-de~Vries Institute for Mathematics, University of Amsterdam, P.~O.~Box 94248, 1090 GE Amsterdam, The Netherlands and ITEP, Moscow, Russia}
\email{P.Dunin-Barkovskiy@uva.nl}

\address{M.~K.: Steklov Mathematical Institute, Ul.~Gubkina 8, 119991 Moscow, Russia and Department of Mathematics, NRU-HSE, Moscow, Russia}
\email{kazarian@mccme.ru}

\address{N.~O.: CAMGSD, Departamento de Matem\'atica,
Instituto Superior T\'ecnico,
Av. Rovisco Pais,
1049-001 Lisboa, Portugal}
\email{norantin@math.ist.utl.pt}

\address{S.~S.: Korteweg-de~Vries Institute for Mathematics, University of Amsterdam, P.~O.~Box 94248, 1090 GE Amsterdam, The Netherlands}
\email{S.Shadrin@uva.nl}

\address{L.~S.: Korteweg-de~Vries Institute for Mathematics, University of Amsterdam, P.~O.~Box 94248, 1090 GE Amsterdam, The Netherlands}
\email{L.Spitz@uva.nl}

\begin{document}
\begin{abstract}
In this paper we give a new proof of the ELSV formula. First, we refine an argument of Okounkov and Pandharipande in order to prove (quasi-)polynomiality of Hurwitz numbers without using the ELSV formula (the only way to do that before used the ELSV formula). Then, using this polynomiality we give a new proof of the Bouchard-Mari\~no conjecture. After that, using the correspondence between the Givental group action and the topological recursion coming from matrix models, we prove the equivalence of the Bouchard-Mari\~no conjecture and the ELSV formula (it is a refinement of an argument by Eynard).
\end{abstract}

\maketitle

\tableofcontents

\section{Introduction}

Hurwitz numbers~$h^\circ_{g,\mu}$ enumerate ramified coverings of the 2-sphere by a connected genus $g$ surface, where the ramification profile over infinity is given by the partition $\mu=(\mu_1,\dots,\mu_\ell)$, there are simple ramifications over $\bn(g,n) = 2g - 2 + l(\mu) + |\mu|$ fixed points, and there are no further ramifications. 

Hurwitz numbers play an important role in the interaction of combinatorics, representation theory of symmetric groups, integrable systems, tropical geometry, matrix models, and intersection theory of the moduli spaces of curves. In this paper we revisit two of the most remarkable properties of Hurwitz numbers. 

The ELSV formula~\cite{ELSV} gives an expression for connected Hurwitz numbers in terms of intersection numbers on the moduli space of curves:
\begin{equation}
\ch_{g,\mu} = \bn(g,n)!
\prod_{i=1}^{\ell(\mu)} \frac{\mu_i^{\mu_i}}{\mu_i !} 
\int_{\oM_{g,\ell(\mu)}}
\frac{\Lambda_g^{\vee}(1)}{\prod_{i=1}^{\ell(\mu)} (1-\mu_i\psi_i)}.
\end{equation}

The Bouchard-Mari\~no conjecture~\cite{BM} (proved by now in several different papers)
is also a relation of Hurwitz numbers to matrix models. Consider the spectral curve
\begin{equation}
x=ye^{-y}
\end{equation}
equipped with the two-point function
\begin{equation}
\frac{dy dy'}{(y-y')^2}.
\end{equation}
Then the $n$-point functions $w_{g,n}$ produced from this data via the matrix model topological recursion~\cite{EO} are equal to 
\begin{equation}
\sum_{\mu_1,\dots,\mu_n} \dfrac{\ch_{g;\mu_1,\dots,\mu_n}}{\bn(g,\mu)!}
     \;\mu_1\dots\mu_n\,x_1^{\mu_1-1}\dots x_n^{\mu_n-1}dx_1\dots dx_n.
\end{equation}

These two statements are known to be equivalent~\cite{Eyn11}, see also~\cite{SSZ13}. We revisit this equivalence in this paper and present this argument in a new way.

Let us describe the existing proofs of both statements. All proofs of the ELSV formula~\cite{ELSV,GV,OP01,Liu} are based, either directly or, as the original one, indirectly, on the computation of the Euler class of the fixed locus of the $\C^*$-action on the space of (relative stable) maps to $\C\mathrm{P}^1$. All mathematically rigorous proofs of the Bouchard-Mari\~no conjecture~\cite{EMS09,MZ} use the ELSV formula and the Laplace transform of the so-called cut-and-join equation for Hurwitz numbers, the basic equation that also allows to reconstruct them recursively. There is one more proof of the Bouchard-Mari\~no conjecture in~\cite{BEMS} that goes through the construction of a matrix model for Hurwitz numbers and a direct derivation of the topological recursion, but it will require plenty of subtle analytic work to make it really mathematically rigorous.  Of course, since the ELSV formula is proved independently, the fact~\cite{Eyn11,SSZ13} that the two statements are equivalent implies the Bouchard-Mari\~no conjecture as well. 

There is still a number of interesting questions on both statements. The first question is whether it is possible to prove the Bouchard-Mari\~no conjecture independently of the ELSV formula. The second question is whether there exists any way to derive the ELSV formula combinatorially, rather than via the computation of the Euler class mentioned above. For example, all Hurwitz numbers can be computed combinatorially, either using the character formula, or, equivalently, using the semi-infinite wedge formalism, or recursively via the cut-and-join equation. On the other hand, the intersection number in the ELSV formula can also be computed combinatorially. Indeed, we can use the Mumford formula~\cite{Mu} for the Chern characters of the Hodge bundle in order to reduce the intersection number in the ELSV formula to intersection numbers of $\psi$-classes, and any intersection number of $\psi$-classes can be computed using the Witten-Kontsevich theorem~\cite{W,K}.  
 The third question, posed e.g.~in~\cite{VakGQT,GJV99}, is the following. The structure of the ELSV formula implies some polynomiality property of Hurwitz numbers, that is
 \begin{equation*}
\ch_{g;\mu_1,\dots,\mu_n} =  \bn(g,n)!\br{\prod_{i=1}^n\dfrac{\mu_i^{\mu_i}}{\mu_i!}}\, P_{g,n}(\mu_1,\dots,\mu_n),
\end{equation*}
where $P_{g,n}(\mu_1,\dots,\mu_n)$ are some polynomials in $\mu_1,\dots,\mu_n$. Though this fact is completely combinatorial, the only way to prove it known up to now is to use the ELSV formula. So, the third question we consider here is whether it is possible to prove this polynomiality in some direct way, without any usage of the ELSV formula. 

This paper provides full answer to all three questions. It is organized in the following way. First, we prove in Section~\ref{sec:polynomiality} the polynomiality of Hurwitz numbers directly from the definition in terms of the semi-infinite wedge formalism. Our argument is a refinement of an argument by Okounkov and Pandharipande in~\cite{OP02}. Then, using the polynomiality property of Hurwitz numbers we are able to derive in Section~\ref{secBM} the Bouchard-Mari\~no conjecture directly from the
 cut-and-join equation. Then, since we have an equivalence of the Bouchard-Mari\~no conjecture and the ELSV formula, we immediately derive the ELSV formula in a new way. In Section~\ref{sec:CEO-Givental} we review the correspondence between the topological recursion and the Givental theory, with a special focus on the 1-dimensional case, and in Section~\ref{sec:BM-ELSV} we provide a (slightly refined) proof of the equivalence of the ELSV formula and the Bouchard-Mari\~no conjecture.
 
\subsection{Acknowledgments} We would like to thank G.~Borot, V.~Bouchard, L.~Che\-khov, B.~Eynard, M.~Mulase, and D.~Zvonkine for many very helpful discussions. N.~O. would like to thank the Korteweg-de Vries Institute for its hospitality during different stages of this project.

P.~D.-B. was supported by free competition grant 613.001.021 of the Netherlands Organization for Scientific Research and partially supported by Ministry of Education and Science of the Russian Federation under contract 8498, by grant NSh-3349.2012.2 and by RFBR grants 13-01-00525 and 13-02-91371-St-a. M.Kazarian was partially supported by RFBR grant 13-01-00383. S.~S. and L.~S. were supported by a Vidi grant of the Netherlands Organization for Scientific Research.



\section{Polynomiality of the Hurwitz numbers} \label{sec:polynomiality}

In this section we prove the following theorem:
\begin{theorem}
\label{ThPol}
The Hurwitz numbers $\ch_{g;\mu_1,\dots,\mu_n}$ for $(g,n)\notin \bs{(0,1),(0,2)}$ can be expressed as follows:
\begin{equation}
\ch_{g;\mu_1,\dots,\mu_n} = (2g+|\mu|+n-2)!\br{\prod_{i=1}^n\dfrac{\mu_i^{\mu_i}}{\mu_i!}}\, P_{g,n}(\mu_1,\dots,\mu_n),
\end{equation}
where $P_{g,n}(\mu_1,\dots,\mu_n)$ is some polynomial in $\mu_1,\dots,\mu_n$.
\end{theorem}

Basically this theorem gives the form of the ELSV formula without specifying the precise formulas for the coefficients. This property (in a bit stronger form) was conjectured in~\cite{GJ99} and then proved in~\cite{GJV99}, with the help of the ELSV formula. Still, the question whether this property can be derived without using the ELSV formula remained open~\cite{VakGQT}. This is precisely what we do here: we prove this statement without using the ELSV formula.

\subsection{Infinite wedge}
In this subsection we recall some basic facts from the theory of the  semi-infinite wedge space following \cite{OP02,OkoPan06,Joh10}.

Let $V$ be an infinite dimensional vector space with a basis labeled by the half integers. Denote the basis vector labeled by $m/2$ by $\underline{m/2}$, so $V = \bigoplus_{i \in \Z + \frac{1}{2}} \underline{i}$.

\begin{definition}
The semi-infinite wedge space $\cV$ is the span of all wedge products of the form
\begin{equation}\label{wedgeProduct}
\underline{i_1} \wedge \underline{i_2} \wedge \cdots
\end{equation}
for any decreasing sequence of half integers $(i_k)$ such that there is an integer $c$ (called the charge) with $i_k + k - \frac{1}{2} = c$ for $k$ sufficiently large. We denote the inner product associated with this basis by~$(\cdot,\cdot)$.

 In this paper, we are mostly concerned with the zero charge subspace $\cV_0\subset\cV$ of the semi-infinite wedge space, which is the space of all wedge products of the form~\eqref{wedgeProduct} such that 
\begin{equation}\label{zeroCharge}
i_k + k = \frac{1}{2}
\end{equation}
for $k$ sufficiently large.
\end{definition}

\begin{remark}
An element of $\cV_0$ is of the form
$$
\underline{\lambda_1 - \frac{1}{2}} \wedge \underline{\lambda_2 - \frac{3}{2}} \wedge \cdots
$$
for some integer partition~$\lambda$. This follows immediately from condition~\eqref{zeroCharge}. Thus, we canonically have a basis for $\cV_0$ labeled by all integer partitions. 
\end{remark}

\begin{notation}
We denote by~$v_\lambda$ the vector labeled by a partition~$\lambda$. The vector labeled by the empty partition is called the vacuum vector and denoted by $|0\rangle = v_{\emptyset} = \underline{-\frac{1}{2}} \wedge \underline{-\frac{3}{2}} \wedge \cdots$. 
\end{notation}

\begin{definition}
If $\mathcal{P}$ is an operator on $\cV_0$, then we define the \emph{vacuum expectation value} of~$\mathcal{P}$ by
$\cor{\mathcal{P}} := \langle 0 |\mathcal{P}|0\rangle$,
where $\langle 0 |$ is the dual of the vacuum vector with respect to the inner product~$(\cdot,\cdot)$, and called the covacuum vector. We will also refer to these vacuum expectation values as (disconnected) \emph{correlators}.
\end{definition}

We now define some operators on the infinite wedge space.

\begin{definition} Let $k$ be any half integer. Then the operator $\psi_k\colon \cV\to\cV$ is defined by
$\psi_k \colon (\underline{i_1} \wedge \underline{i_2} \wedge \cdots) \ \mapsto \ (\underline{k} \wedge \underline{i_1} \wedge \underline{i_2} \wedge \cdots)$. It increases the charge by $1$.

The operator $\psi_k^*$ is defined to be the adjoint of the operator $\psi_k$ with respect to the inner product~$(\cdot,\cdot)$.
\end{definition}

\begin{definition}
The normally ordered products of $\psi$-operators are defined in the following way
\begin{equation}
E_{ij}:={:}\psi_i \psi_j^*{:}\ := \begin{cases}\psi_i \psi_j^*, & \text{ if } j > 0 \\
-\psi_j^* \psi_i & \text{ if } j < 0\ .\end{cases} 
\end{equation}
This operator does not change the charge and can be restricted to $\cV_0$. Its action on the basis vectors $v_\lambda$ can be described as follows: ${:}\psi_i \psi_j^*{:}$ checks if $v_\lambda$ contains $\underline{j}$ as a wedge factor and if so replaces it by $\underline{i}$. Otherwise it yields~$0$. In the case $i=j > 0$, we have ${:}\psi_i \psi_j^*{:}(v_\lambda) = v_\lambda$ if $v_\lambda$ contains $\underline{j}$ and $0$ if it does not; in the case  $i=j < 0$, we have ${:}\psi_i \psi_j^*{:}(v_\lambda) = - v_\lambda$ if $v_\lambda$ does not contain $\underline{j}$ and $0$ if it does. These are the only two cases where the normal ordering is important.
\end{definition}

\begin{notation}\label{notationZeta}
We denote by $\zeta(z)$ the function $e^{z/2} - e^{-z/2}$.
\end{notation}

\begin{definition}
Let $n \in \Z$ be any integer. We define an operator $\mathcal{E}_n(z)$ depending on a formal variable~$z$ by
\begin{align*}
\mathcal{E}_n(z) &= \sum_{k \in \Z + \frac{1}{2}} e^{z(k - \frac{n}{2})} E_{k-n,k} + \frac{\delta_{n,0}}{\zeta(z)}  .
\end{align*}
\end{definition}

Note that 
\begin{equation}\label{commcEE}
  \left[\cE_a(z),\cE_b(w)\right] =
\zeta\left(\det  \left[
\begin{smallmatrix}
  a & z \\
b & w
\end{smallmatrix}\right]\right)
\,
\cE_{a+b}(z+w)
\end{equation}
and
\begin{equation}
\cE_0(z)\big|0\big\rangle = \dfrac{1}{\zeta(z)}\big|0\big\rangle
\end{equation}
and also that
\begin{equation}
\cE_k(z)\big|0\big\rangle = 0,\quad k>0 .
\end{equation}

\begin{definition}
In what follows we will use the following operator:
$${\mathcal F}_2 := \sum_{k\in\Z+\frac12} \frac{k^2}{2} E_{k,k} .$$
\end{definition}

\begin{definition}
We will also need the following operators:
$$\alpha_k:=\mathcal{E}_k(0), \quad k\neq 0 .$$
\end{definition}



\subsection{Hurwitz numbers in the infinite wedge formalism}

By  $\dih_{g,\mu}$ we denote the Hurwitz numbers for possibly disconnected covering surfaces. 
The character formula for the disconnected Hurwitz numbers $\dih_{g,\mu}$ implies that (see e.g. \cite{OP02}) 
\begin{equation}
\label{jkk}
\dih_{g,\mu} =
\cor{e^{\alpha_1} \mathcal{F}_2^{\bn(g,\mu)} \prod_{i=1}^{l(\mu)} \dfrac{\alpha_{-\mu_i}}{\mu_i}} .
\end{equation}
Here $l(\mu)$ denotes the number of parts of $\mu$, and 
\begin{equation}
\label{bdef}
\bn(g,\mu) := 2g+|\mu|+l(\mu)-2 .
\end{equation} 
Note the difference between our disconnected Hurwitz numbers $\dih_{g,\mu}$ and the ones in \cite{OP02} (which are denoted by $\mathsf{C}_g(\mu)$ there). The difference is in a factor of $\left|\mathrm{Aut}(\mu)\right|$, the number of automorphisms of the partition.

\begin{definition}
Define the genus-generating functions for the disconnected Hurwitz numbers and for the connected ones as well:
\begin{equation}
\dih_{\mu}(u) := \sum_{g=0}^{\infty} \dfrac{u^{2g-2}}{\bn(g,\mu)!}\dih_{g,\mu} ,
\end{equation}
\begin{equation}
\ch_{\mu}(u) := \sum_{g=0}^{\infty} \dfrac{u^{2g-2}}{\bn(g,\mu)!}\ch_{g,\mu} .
\end{equation}
\end{definition}
They are related to each other through the inclusion-exclusion formula.

We have
\begin{align}
\label{hurwexpr}
\dih_{\mu}(u) & = u^{-|\mu|-l(\mu)}\cor{e^{\alpha_1}e^{u\mathcal{F}_2}\prod_{i=1}^{l(\mu)} \dfrac{\alpha_{-\mu_i}}{\mu_i}}\\ \nonumber
&= u^{-|\mu|-l(\mu)}\cor{e^{\alpha_1}e^{u\mathcal{F}_2}\br{\prod_{i=1}^{l(\mu)} \dfrac{\alpha_{-\mu_i}}{\mu_i}}e^{-u\mathcal{F}_2}e^{-\alpha_1}}\\ \nonumber
&= u^{-|\mu|-l(\mu)}\cor{\prod_{i=1}^{l(\mu)}\br{e^{\alpha_1}e^{u\mathcal{F}_2}\dfrac{\alpha_{-\mu_i}}{\mu_i}e^{-u\mathcal{F}_2}e^{-\alpha_1}}} .
\end{align}
The second equality holds since $e^{-u\mathcal{F}_2}$ and $e^{-\alpha_1}$ fix the vacuum vector.

\subsection{$\A$-operators}
Now, following~\cite{OP02}, we introduce certain operators that we use later on to rewrite the formula for Hurwitz numbers.
\begin{definition}
Define
\begin{equation}\label{defA}
\A(a,b) := \br{\dfrac{\zeta(b)}{b}}^a \,
\sum_{k\in \Z} \frac{\zeta(b)^{k}}{(a+1)_k} \, \cE_k(b),
\end{equation}
where $a$ and $b$ are parameters and we use the standard notation:
\begin{equation}
(a+1)_k = \frac{(a+k)!}{a!} =
\begin{cases}
 (a+1) (a+2) \cdots (a+k) \,, & k\ge 0 \,,  \\
 (a (a-1) \cdots (a+k+1))^{-1}  \,, & k \le 0 \,.
\end{cases}
\end{equation}
If $a\ne 0,1,2,\dots$, the sum in \eqref{defA} is
infinite in both directions. If $a$ is a non-negative
integer, the summands with $k \le - a - 1$ in \eqref{defA}
vanish.

Note that Proposition 3 of \cite{OP02} implies that the correlator
\begin{equation}
\cor{\A(z_1,uz_1)\dots\A(z_n,uz_n)}
\end{equation}
is well-defined for all $\br{z_1,\dots,z_n}\in \Omega \subset \mathbb{C}^n$ and sufficiently small $u$, where
\begin{equation}
\Omega = \br{\br{z_1,\dots,z_n} \Bigg| |z_k| > \sum_{i=1}^{k-1} |z_i|,\; k=1,\dots,n} .
\end{equation}

\end{definition}
\begin{definition}
Define the \emph{connected correlator} of $\A$-operators
$$\ccor{\A(z_1,uz_1)\dots\A(z_n,uz_n)}$$
through the disconnected ones via the inclusion-exclusion formula.
\end{definition}

\begin{proposition}
\begin{equation}
\label{hconexpr}
\ch_{g;\mu_1\dots\mu_n} = \bn(g,\mu)!\,\prod_{i=1}^n\br{\dfrac{\mu_i^{\mu_i-1}}{\mu_i!}}[u^{2g-2+n}]\ccor{\A(\mu_1,u\mu_1)\dots\A(\mu_n,u\mu_n)}
\end{equation}
where $[u^{2g-2+n}]\ccor{\A(\mu_1,u\mu_1)\dots\A(\mu_n,u\mu_n)}$ stands for the coefficient of $u^{2g-2+n}$ in $\ccor{\A(\mu_1,u\mu_1)\dots\A(\mu_n,u\mu_n)}$.
\end{proposition}

\begin{proof} The main part of the proof follows~\cite{OP02}. Note that
\begin{equation}
e^{u \mathcal{F}_2} \, \alpha_{-m} \, e^{-u \mathcal{F}_2} = \cE_{-m}(um)
\end{equation}
which is easy to see since $\mathcal{F}_2$ acts diagonally. From the commutation relations for $\cE_i$ we see that
\begin{equation}
e^{\alpha_1} \, \cE_{-m}(s) \, e^{-\alpha_1}  =
\frac{\zeta(s)^{m}}{m!}
\sum_{k\in \Z} \frac{\zeta(s)^{k}}{(m+1)_k} \, \cE_k(s) .
\end{equation}
The previous two formulas imply the following, for $m\in\bs{1,2,3,\dots}$ (Lemma 2 of \cite{OP02}):
\begin{equation}
e^{\alpha_1} \, e^{u \mathcal{F}_2} \, \alpha_{-m} \, e^{-u \mathcal{F}_2} \, e^{-\alpha_1} =
\frac{u^m \, m^m}{m!} \, \A(m,um) .
\end{equation}
Now we can rewrite formula \eqref{hurwexpr} as
\begin{equation}
h_{\mu_1\dots\mu_n}(u) = u^{-n}\prod_{i=1}^n\br{\dfrac{\mu_i^{\mu_i-1}}{\mu_i!}}\cor{\A(\mu_1,u\mu_1)\dots\A(\mu_n,u\mu_n)} .
\end{equation}
Recall that the connected Hurwitz numbers can be expressed though the disconnected ones with the help of the inclusion-exclusion formula.
Since the relation between connected and disconnected Hurwitz numbers is the same as the one between connected and disconnected correlators, we have:
\begin{equation}
\ch_{\mu_1\dots\mu_n}(u) = u^{-n}\prod_{i=1}^n\br{\dfrac{\mu_i^{\mu_i-1}}{\mu_i!}}\ccor{\A(\mu_1,u\mu_1)\dots\A(\mu_n,u\mu_n)} .
\end{equation}
Comparing the coefficients in front of the same powers of $u$ on the right hand side and on the left hand side we directly obtain the statement of the proposition.
\end{proof}

Now we see that in order to prove Theorem~\ref{ThPol} we only have to show that 
expressions $[u^{2g-2+n}]\ccor{\A(\mu_1,u\mu_1)\dots\A(\mu_n,u\mu_n)}$ are polynomial in $\mu_1,\dots,\mu_n$.


\subsection{Further properties of $\A$-operators} 

In this subsection we modify  an expression for the connected correlators of $\A$-operators in order to exclude possible so-called \emph{unstable terms}.

\begin{definition}
Let $\A_k$ be the coefficients of the expansion of
the operator $\A(z,uz)$ in powers of $z$:
\begin{equation}
  \label{defAk}
  \A(z,uz)= \sum_{k\in \Z} \A_k\, z^k \,.
\end{equation}
\end{definition}

We will use the following theorem, due to Okounkov and Pandharipande:
\begin{theorem}[Okounkov-Pandharipande, \cite{OP02}]
\label{ThA}
\begin{equation}
  \label{cA2}
  \left[\A_k, \A_l\right] = (-1)^l \delta_{k+l-1} \,.
\end{equation}
\end{theorem}

\begin{definition}
Define
\begin{equation}
  \label{defAplus}
  \A_+(z,uz):= \sum_{k=1}^\infty \A_k\, z^k \,.
\end{equation}
\end{definition}

\begin{notation}
For any operator $\mathcal{P}(u)$ define
\begin{align}
\cor{\mathcal{P}(u)}_k &:= [u^k]\cor{\mathcal{P}(u)} & (\mathrm{the\ coefficient\ of\ } u^k\  \mathrm{in}\ \cor{\mathcal{P}(u)}) \\ \nonumber
\ccor{\mathcal{P}(u)}_k &:= [u^k]\ccor{\mathcal{P}(u)} & (\mathrm{the\ coefficient\ of\ } u^k\  \mathrm{in}\ \ccor{\mathcal{P}(u)})
\end{align}
\end{notation}

\begin{definition}
We denote by $\mathcal{Y}_{n,k}$ the set of $\bs{1,\dots,n}$ Young tableaux (i.~e.~Young diagrams of size $n$ with each box labeled by a number from $1$ to $n$ such that no two boxes are labeled by the same number) with certain conditions and additional row labels.

Namely, let $y$ be such a tableau. Let $c_{i,j}(y)$ be the number in the $i$-th row and $j$-th column. Let $h(y)$ be the number of rows, and let $l_i(y)$ be the length of the $i$-th row. Now we are ready to describe the conditions.

First, the numbers in the rows should be ascending, i.~e.~ for any $i$ and for any $j_1 < j_2$ we have $c_{i,j_1}(y) < c_{i,j_2}(y)$. Second, the numbers in the first column that correspond to rows of the same length should be ascending, i.~e., 
if $l_{i_1}(y)=l_{i_2}(y)$ and $i_1<i_2$, then $c_{i_1,1}(y) < c_{i_2,1}(y)$.

By $\lambda_i(y)\in\bs{-1,0,1,\dots}$ we denote additional labels that are assigned to all rows, and we require that 
$\sum_{i=1}^{h(y)}\lambda_i(y) = k$.
\end{definition}

Note that there is a one-to-one correspondence between the elements of $\mathcal{Y}_{n,k}$ and the terms in the expression for a disconnected correlator through the connected ones (the ``inverse'' inclusion-exclusion formula). 
Rows in $y$ correspond to individual connected correlators in the product, while labels $\lambda$ correspond to the Euler characteristics of these connected correlators. This can be expressed through the following formula:
\begin{align}
\label{Aexp}
 &\cor{\A(z_1,uz_1)\dots\A(z_n,uz_n)}_k\\ \nonumber
 &=\sum_{y \in \mathcal{Y}_{n,k}}\prod_{i=1}^{h(y)}\ccor{\A(z_{c_{i,1}(y)},uz_{c_{i,1}(y)})\dots\A(z_{c_{i,l_i(y)}(y)},uz_{c_{i,l_i(y)}(y)})}_{\lambda_i(y)} .
\end{align}
The terms in this sum that contain either $\ccor{\A(z_i,uz_i)}_{-1}$ or $\ccor{\A(z_i,uz_i)\A(z_j,uz_j)}_{0}$ are called \emph{unstable terms}. If we exclude all unstable terms, we obtain the following expression.

\begin{proposition} We have:
\begin{align}
\label{Aplusexp}
 &\cor{\A_+(z_1,uz_1)\dots\A_+(z_n,uz_n)}_k\\ \nonumber
 &=\sum_{y \in \widetilde{\mathcal{Y}}_{n,k}}\prod_{i=1}^{h(y)}\ccor{\A(z_{c_{i,1}(y)},uz_{c_{i,1}(y)})\dots\A(z_{c_{i,l_i(y)}(y)},uz_{c_{i,l_i(y)}(y)})}_{\lambda_i(y)}.
\end{align}
Here \begin{equation}
{\mathcal{Y}}_{n,k} = \bs{y\in\mathcal{Y}_{n,k}\big|\, l_i(y)=1 \Rightarrow \lambda_i(y)\neq -1,\; l_i(y)=2 \Rightarrow \lambda_i(y)\neq 0}
\end{equation}
In other words, $\cor{\A_+(z_1,uz_1)\dots\A_+(z_n,uz_n)}_k$ is equal to $\cor{\A(z_1,uz_1)\dots\A(z_n,uz_n)}_k$ with all the unstable terms dropped.
\end{proposition}

\begin{proof}
Let us first compute the unstable factors, i.~e.~the
genus-zero one- and two-point connected correlators.

Note that
\begin{equation}
\label{1pte}
\cor{\A(z,uz)} = \dfrac{1}{uz}+\dfrac{z(z-1)}{24}u+O(u^2).
\end{equation}
This directly implies the following formula for the genus-zero one-point correlator:
\begin{equation}
\label{1ptce}
\ccor{\A(z,uz)}_{-1} = \dfrac{1}{z} .
\end{equation}

The definition of the operator $\A$ implies that
\begin{equation}
\label{lav}
\langle 0|\A(z,uz) = \dfrac{1}{uz}\langle 0| + \langle 0| \A_+(z,uz) .
\end{equation}
The definition of the two-point connected correlators together with formulas \eqref{1pte}, \eqref{lav} and \eqref{cA2} implies the following formula for the genus-zero two-point connected correlator:
\begin{align}
\label{2ptce}
\ccor{\A(z_1,uz_1)\A(z_2,uz_2)}_{0}
& = \cor{\A(z_1,uz_1)\A(z_2,uz_2)}_0  -\cor{\A(z_1,uz_1)}_{-1}\cor{\A(z_2,uz_2)}_{1}
 \\ \notag
 & \phantom{=\ } -\cor{\A(z_1,uz_1)}_{1}\cor{\A(z_2,uz_2)}_{-1}\\ \nonumber
&=\cor{\A_+(z_1,uz_1)\A(z_2,uz_2)}_0-\cor{\A(z_1,uz_1)}_{1}\cor{\A(z_2,uz_2)}_{-1}\\ \nonumber
&=\cor{\A(z_2,uz_2)\A_+(z_1,uz_1)}_0+z_1\sum_{k=0}^\infty(-1)^k\br{\dfrac{z_1}{z_2}}^k
\\ \notag &
\phantom{=\ } 
-\cor{\A(z_2,uz_2)}_{-1}\cor{\A(z_1,uz_1)}_{1}\\ \nonumber
&=\cor{\A_+(z_2,uz_2)\A_+(z_1,uz_1)}_0+z_1\sum_{k=0}^\infty(-1)^k\br{\dfrac{z_1}{z_2}}^k\\ \nonumber
&=z_1\sum_{k=0}^\infty(-1)^k\br{\dfrac{z_1}{z_2}}^k .
\end{align}

Now we prove the statement of the proposition by induction over the number of operators $n$ in the correlator on the left hand side. From the definition of the operator $\A$ it is easy to see that the statement holds for $n=1$. Suppose that it holds for the correlator of any number of operators less than $n$. We will prove that it holds for $n$ operators. 

Taking into account \eqref{lav}, \eqref{cA2}, \eqref{1ptce} and \eqref{2ptce} we see that
\begin{align}
\label{Apluse}
&
\cor{\A(z_1,uz_1)\dots  \A(z_n,uz_n)}_k 
\\ \notag &
= \dfrac{1}{z_1} \cor{\A(z_2,uz_2)\dots\A(z_n,uz_n)}_{k+1}
\\ \notag &
\phantom{ = }\ +\cor{\A_+(z_1,uz_1)\A(z_2,uz_2)\dots\A(z_n,uz_n)}_k
\\ \nonumber &
=\ccor{\A(z_1,uz_1)}_{-1} \cor{\A(z_2,uz_2)\dots\A(z_n,uz_n)}_{k+1}\\ \nonumber
&\phantom{ = }\ + z_1\sum_{k=0}^\infty(-1)^k\br{\dfrac{z_1}{z_2}}^k \cor{\A(z_3,uz_3)\dots\A(z_n,uz_n)}_{k}\\ \nonumber
&\phantom{ = }\ +\cor{\A(z_2,uz_2)\A_+(z_1,uz_1)\A(z_3,uz_3)\dots\A(z_n,uz_n)}_{k}\\ \nonumber
&=\ccor{\A(z_1,uz_1)}_{-1} \cor{\A(z_2,uz_2)\dots\A(z_n,uz_n)}_{k+1}\\ \nonumber
&\phantom{ = }\ + \ccor{\A(z_1,uz_1)\A(z_2,uz_2)}_0\cor{\A(z_3,uz_3)\dots\A(z_n,uz_n)}_{k}\\ \nonumber
&\phantom{ = }\ + \ccor{\A(z_2,uz_2)}_{-1} \cor{\A_+(z_1,uz_1)\A(z_3,uz_3)\dots\A(z_n,uz_n)}_{k+1}\\ \nonumber
&\phantom{ = }\ + \cor{\A_+(z_1,uz_1)\A_+(z_2,uz_2)\A(z_3,uz_3)\dots\A(z_n,uz_n)}_{k} .
\end{align}
We continue with the same computation (replacing the leftmost operator $\A$ with $\A_+$ and commuting it to the right, collecting the emerging coefficients in the unstable correlators), finally arriving at the following expression.
\begin{align}\label{Apluse-2}
&\cor{\A(z_1,uz_1)\dots\A(z_n,uz_n)}_k =\cor{\A_+(z_1,uz_1)\dots\A_+(z_n,uz_n)}_{k}\\ \nonumber
&\phantom{ = }\ +
\sum_{p=3}^{n-1}\,\sum_{q=0}^{[\frac{n-p}2]}\sum_{y \in \widehat{\mathcal{Y}}^{p,q}_{n,k}}\cor{\A_+(z_{c_{1,1}(y)},uz_{c_{1,1}(y)})\dots\A_+(z_{c_{1,p}(y)},uz_{c_{1,p}(y)})}_{k+h(y)-q-1}\\ \nonumber
&\phantom{ = }\ \phantom{ = }\ \times\prod_{i=2}^{q+1}\ccor{\A(z_{c_{i,1}(y)},uz_{c_{i,1}(y)})\A(z_{c_{i,2}(y)},uz_{c_{i,2}(y)})}_{0}
\prod_{i=q+2}^{h(y)}\ccor{\A(z_{c_{i,1}(y)},uz_{c_{i,1}(y)})}_{-1}\\ \nonumber
&\phantom{ = }\ +
\sum_{q=0}^{[\frac{n-2}{2}]}\sum_{y \in \widehat{\mathcal{Y}}^{2,q}_{n,k}}\cor{\A_+(z_{c_{s(y),1}(y)},uz_{c_{s(y),1}(y)})\A_+(z_{c_{s(y),2}(y)},uz_{c_{s(y),2}(y)})}_{k+h(y)-q-1}\\ \nonumber
&\phantom{ = }\ \phantom{ = }\ \times\prod^{q+1}_{\substack{i=1\\ i\not=s(y)}}\ccor{\A(z_{c_{i,1}(y)},uz_{c_{i,1}(y)})\A(z_{c_{i,2}(y)},uz_{c_{i,2}(y)})}_{0}
\prod_{i=q+2}^{h(y)}\ccor{\A(z_{c_{i,1}(y)},uz_{c_{i,1}(y)})}_{-1}\\ \nonumber
&\phantom{ = }\ +
\sum_{q=0}^{[\frac{n-1}2]}\sum_{y \in \widehat{\mathcal{Y}}^{1,q}_{n,k}}\cor{\A_+(z_{c_{s(y),1}(y)},uz_{c(y)_{s(y),1}(y)})}_{k+h(y)-q-1}\\ \nonumber
&\phantom{ = }\ \phantom{ = }\ \times\prod_{i=1}^{q}\ccor{\A(z_{c_{i,1}(y)},uz_{c_{i,1}(y)})\A(z_{c_{i,2}(y)},uz_{c_{i,2}(y)})}_{0}
\prod^{h(y)}_{\substack{i={q+1}\\ i\not = s(y) }}\ccor{\A(z_{c_{i,1}(y)},uz_{c_{i,1}(y)})}_{-1}
\end{align}
Here $\widehat{\mathcal{Y}}^{p,q}_{n,k}$ contains all elements $y$ of $\mathcal{Y}_{n,k}$ such that there is precisely one row of length $p$ labeled by $k+h(y)-q-1$, $q$ rows of length $2$ labeled by $0$, and all other rows are of length $1$ and labeled by $-1$. $s(y)$ stands for the position of the row with $p$ elements labeled by $k+h(y)-q-1$. If $p=2$ and $k+h(y)-q-1=0$ or $p=1$ and $k+h(y)-q-1=-1$ one cannot determine $s(y)$ in this way, but this is not a problem since, due to the fact that $\cor{\A_+(z,uz)}_{-1}=0$ and $\cor{\A_+(z_{1},uz_{1})\A_+(z_{2},uz_{2})}_0=0$, the corresponding term vanishes in any case. Also note that, obviously, for $p\geq 3$ we have $s(y)=1$.

Note that the right hand side of formula \eqref{Apluse-2} is equal to the correlator
\begin{equation}
\cor{\A_+(z_1,uz_1)\dots\A_+(z_n,uz_n)}_{k}
\end{equation}
plus all possible unstable terms entering exactly once, since, by the induction hypothesis, the correlators of less than $n$ operators $\A_+$ are equal to sums of all possible stable terms. This means that upon moving these terms to the left hand side and subtracting them from $\cor{\A(z_1,uz_1)\dots\A(z_n,uz_n)}_k$ we get precisely all possible stable terms. This proves the proposition.
\end{proof}


\subsection{Polynomiality}

In this subsection we establish polynomiality of some correlators, and this allows us to complete the proof of Theorem \ref{ThPol}.

\begin{proposition}
\label{Aplusprop}
The series
\begin{equation}
\displaystyle\dfrac{\cor{\A_+(z_1,uz_1)\dots\A_+(z_n,uz_n)}_k}{z_1 \cdots z_n}
\end{equation}
for $(n,k)\notin\bs{(1,-1),(2,0)}$ is a symmetric polynomial in $z_1,\dots,z_n$.
\end{proposition}
\begin{proof}
From the definition of $\A_+$ it is easy to see that for every $i$ the power of $z_i$ in the series $\displaystyle\cor{\A_+(z_1,uz_1)\dots\A_+(z_n,uz_n)}_k$ is bounded from below by $1$. From \eqref{cA2} it is clear that this series is symmetric in $z_1,\dots,z_n$. Let us prove that, for fixed $k$, the power of $z_n$ in this series is bounded from above, following the proof of Proposition 9 of \cite{OP02}.

Note that 
\begin{equation} \label{re1}
  \cor{\cE_{k_1}(uz_1)\dots \cE_{k_n}(uz_n)} = \cor{\dfrac{\cE_{k_1}(uz_1)}{u^{k_1}}\dots \dfrac{\cE_{k_n}(uz_n)}{u^{k_n}}},
\end{equation}
holds since the correlator vanishes unless $\sum k_i = 0$.

Let us apply this transformation to the correlator of $\A$ operators:
\begin{equation}
\label{acor}
\cor{\A(z_1,uz_1) \dots \A(z_n,u z_n)} = \cor{\widetilde{\A}(z_1,uz_1) \dots \widetilde{\A}(z_{n},u z_{n})} .
\end{equation}
Here $\widetilde{\A}$ stands for the operator $\A$ where the substitution $\cE_k \mapsto u^{-k}\cE_k$ was made.
Note that each term in each $\widetilde{\A}$ is then regular and non-vanishing at $u=0$, except for the term $\dfrac{1}{\zeta(uz)}$ coming from $\cE_0$, which has a simple pole. Let us write the following:
\begin{align}
&\widetilde{\A}(z_n,uz_n)\,\big|0\big\rangle\\ \nonumber
&=\br{\dfrac{\zeta(uz_n)}{uz_n}}^{z_n}\,\sum_{k\in \Z} \dfrac{\zeta(uz_n)^{k}}{(z_n+1)_k} \, \dfrac{\cE_k(uz_n)}{u^k}\big|0\big\rangle\\ \nonumber
&=\br{\dfrac{\zeta(uz_n)}{uz_n}}^{z_n}\sum_{k=0}^{\infty} \dfrac{u^k}{\zeta(uz_n)^{k}} z_n\dots(z_n-k+1) \, \cE_{-k}(uz_n)\big|0\big\rangle\\ \nonumber
&=\br{\dfrac{\zeta(uz_n)}{uz_n}}^{z_n}\sum_{k=0}^{\infty} \br{\dfrac{uz_n}{\zeta(uz_n)}}^k\br{1-\dfrac{1}{z_n}}\dots\br{1-\dfrac{k-1}{z_n}} \, \cE_{-k}(uz_n)\big|0\big\rangle .
\end{align}
It is easy to see that $z_n$ and $u$ enter this expression in such a way that for all terms with a fixed power of $u$ the power of $z_n$ is bounded from above. Since in \eqref{acor} this expression is multiplied by operators $\widetilde{\A}(z_i,uz_i),\, i\in\bs{1,\dots,n-1}$, which have at most simple poles in $u$, the whole correlator \eqref{acor} is bounded from above in powers of $z_n$, for a fixed power of $u$.

From the definition of $\A_+$ it is clear that the fact that the power of $z_n$ in
\begin{equation*}
\cor{\A(z_1,uz_1) \dots \A(z_n,u z_n)}_k
\end{equation*}
is bounded from above for a fixed $k$ immediately implies that the power of $z_n$ in
\begin{equation*}
\cor{\A_+(z_1,uz_1) \dots \A_+(z_n,u z_n)}_k
\end{equation*}
is bounded from above for a fixed $k$ as well.

The symmetricity of $\displaystyle\cor{\A_+(z_1,uz_1)\dots\A_+(z_n,uz_n)}_k$ then implies that for fixed $k$ the power of $z_i$ in this expression is bounded from above for any $i$, which implies that  the series
\begin{equation}
\cor{\A_+(z_1,uz_1)\dots\A_+(z_n,uz_n)}_k
\end{equation}
is polynomial in $z_1,\dots,z_n$, which, in turn (if one takes into account the fact that for all $i$ the power of $z_i$ in the above expression is bounded from below by 1), leads to the fact that the series
\begin{equation}
\dfrac{\cor{\A_+(z_1,uz_1)\dots\A_+(z_n,uz_n)}_k}{z_1 \cdots z_n}
\end{equation}
is polynomial in $z_1,\dots,z_n$.
\end{proof}

\begin{proposition}
\label{Praconpol}
For $(n,k)\notin\bs{(1,-1),(2,0)}$ the series
\begin{equation}
\displaystyle\dfrac{\ccor{\A(z_1,uz_1)\dots\A(z_n,uz_n)}_k}{z_1 \cdots z_n}
\end{equation}
is a symmetric polynomial in $z_1,\dots,z_n$.
\end{proposition}
\begin{proof}
Let us prove the statement of this proposition by induction in $n$, the number of operators in the correlator. It is clear that for $n=1$ the statement holds. Suppose that it holds for any number of operators less than $n$. We will prove that it then holds for $n$ operators as well.

Formula \eqref{Aplusexp} can be rewritten as
\begin{align}
\label{Aconexp}
 &\dfrac{\ccor{\A(z_1,uz_1)\dots\A(z_n,uz_n)}_k}{z_1 \cdots z_n} =\dfrac{\cor{\A_+(z_1,uz_1)\dots\A_+(z_n,uz_n)}_k}{z_1 \cdots z_n}\\ \nonumber
 &-\sum_{y \in \widetilde{\mathcal{Y}}'_{n,k}}\prod_{i=1}^{h(y)}\dfrac{\ccor{\A(z_{c_{i,1}(y)},uz_{c_{i,1}(y)})\dots\A(z_{c_{i,l_i(y)}(y)},uz_{c_{i,l_i(y)}(y)})}_{\lambda_i(y)}}{z_{c_{i,1}(y)} \cdots z_{c_{i,l_i(y)}(y)}} .
\end{align}
Here, naturally, $\widetilde{\mathcal{Y}}'_{n,k}$ is equal to $\widetilde{\mathcal{Y}}_{n,k}$ with the single-row Young tableau thrown away.

By Proposition \ref{Aplusprop}, the first term on the right hand side of \eqref{Aconexp} is polynomial in $z_1,\dots,z_n$. By induction hypothesis, all the terms in the sum on the right hand side of \eqref{Aconexp} are polynomial as well, since they are finite products of connected correlators of the lower number of operators (and, by definition of $\widetilde{\mathcal{Y}}'_{n,k}$, correlators with $(n_i,k_i) \in \bs{(1,-1),(2,0)}$ never appear). This implies the statement of the proposition.
\end{proof}

Taking into account formula \eqref{hconexpr}, we see that Proposition \ref{Praconpol} directly implies the statement of Theorem \ref{ThPol}.



\section{Proof of the Bouchard-Mari\~no conjecture}
\label{secBM}

In the present section we give a new proof of the Bouchard-Mari\~no conjecture using the polynomiality result from the previous section and not using the ELSV formula.

This conjecture was already proved in~\cite{BEMS} and~\cite{EMS09}. The first of these papers provides a ``physical'' proof through the study of the corresponding matrix model. Unfortunately, we were not able to attribute precise mathematical meaning to all of the statements of that paper (see \cite{SSZ13} for a related discussion). In the second paper the Bouchard-Mari\~no formula is derived directly from the known cut-and-join recursion relation for Hurwitz numbers, with the help of the ELSV formula. 

Here we follow the ideas of the proof of~\cite{EMS09} presenting them in a simplified way, with one essential modification: we do not use the ELSV formula in this proof, using instead just the polynomiality property.

\subsection{Generating function for Hurwitz numbers}

Let us introduce the generating function for the connected Hurwitz numbers $\ch_{g;\mu}$ in the following way:
\begin{equation}
\cH_{g,n} := \sum_{\mu_1,\dots,\mu_n \in \bs{1,2,\dots}} \dfrac{\ch_{g;\mu_1,\dots,\mu_n}}{\bn(g,\mu)!}\,x_1^{\mu_1}\dots_n^{\mu_n} .
\end{equation}
Theorem \ref{ThPol} implies that, for $(g,n) \notin \bs{(0,1),(0,2)}$,
\begin{equation}
\label{hurwcoefexpr}
\cH_{g,n} = \sum_{\substack{k_1,\dots,k_n\in\\ \bs{0,1,\dots,K_{g,n}}}} c_{k_1\dots k_n} \prod_{i=1}^n\sum_{\mu_i=1}^{\infty}\dfrac{\mu_i^{\mu_i+k_i}}{\mu_i!}x_i^{\mu_i},
\end{equation}
where $c_{k_1\dots k_n}$ are the coefficients of the polynomials $P_{g,n}$ from Theorem \ref{ThPol}, and $K_{g,n}$ is the highest power appearing in $P_{g,n}$.

Define
\begin{equation}
\xv_k(x) := \sum_{m=1}^{\infty}\dfrac{m^{m+k}}{m!}x^{m} .
\end{equation}
Now we can rewrite \eqref{hurwcoefexpr} as
\begin{equation}
\cH_{g,n} = \sum_{\substack{k_1,\dots,k_n\in\\ \bs{0,1,\dots,K_{g,n}}}} c_{k_1\dots k_n} \prod_{i=1}^n\xv_{k_i}(x_i) .
\end{equation}

Consider the following change of variables:
\begin{equation}
\label{xt}
x_i=\Bigl(1+\frac1{t_i}\Bigr)\,e^{-1-\frac1{t_i}} .
\end{equation}
We see that the generating function $H_{g,n}$ is a polynomial in variables $t_i$ (in all but two
`unstable' cases when $g=0$ and $n\le2$) after the above substitution (we treat this substitution as a power series expansion at the point $t_i=-1$). For the unstable cases we have
\begin{equation}
\begin{aligned}
\cH_{0,1}&=\sum_{a=1}^\infty\frac{a^{a-2}}{a!}x_1^a=\xv_{-2}(x_1)=
\frac12-\frac{1}{2t_1^2},\\
\cH_{0,2}&=\sum_{a,b}\frac{a^a}{a!}\frac{b^b}{b!}\frac{x_1^ax_2^b}{a+b}
 =\log\left(\frac{\frac1{t_2+1}-\frac1{t_1+1}}{\frac1{x_1}-\frac1{x_2}}\right) .
\end{aligned}\label{H0102}
\end{equation}

The formula of Bouchard and Mari\~no is a
recursion relation for these polynomials. In order to present it
in a more closed form it is convenient to introduce another family of
polynomials $W_{g,n}(t_1,\dots,t_n)$ obtained by the above
substitution from the series
$$\Bigl(\prod x_k\d_{x_k}\Bigr)\cH_{g,n}=
 \sum_{\mu_1,\dots,\mu_n}
 \frac{\ch_{g;\mu_1,\dots,\mu_n}}{\bn(g,\mu)!}
     \;\mu_1\dots\mu_n\,x_1^{\mu_1}\dots x_n^{\mu_n},
$$
i.~e., for $(g,n)\notin\bs{(0,1),(0,2)}$,
\begin{equation}
\label{Wset}
W_{g,n}(t_1,\dots,t_n)=\sum_{\substack{k_1,\dots,k_n\in\\ \bs{0,1,\dots,K_{g,n}}}} c_{k_1\dots k_n}
 \prod_{i=1}^n\xv_{k_i+1}(t_i).
\end{equation}
In the unstable cases we define the functions $W_{g,n}$ by setting explicitly
\begin{align}
\label{wunst1}
 W_{0,1}(t_1)&=0,\\ \label{wunst2}
 W_{0,2}(t_1,t_2)&=
 \frac{t_1^2(t_1+1)t_2^2(t_2+1)}{(t_2-t_1)^2} .
 \end{align}

Define also auxiliary functions $\tW_{g,n}(u,v;t_2,\dots,t_n)$ by
\begin{align*}
 \tW_{g,n}(u,v;t_{L'}):= & W_{g-1,n+1}(u,v,t_{L'}) \\ & +\sum_{g_1+g_2=g}
  \sum_{A\sqcup B=L'}W_{g_1,|A|+1}(u,t_A)W_{g_2,|B|+1}(v,t_B) .
\end{align*}
We denote here by $L'=\{2,\dots,n\}$ the index set,
$t_{L'}=(t_2,\dots,t_n)$; the summation is taken over the set of
all possible partitions of the index set into a disjoint union of
two subsets, $A$ and~$B$.

 \begin{theorem}[Bouchard-Mari\~no conjecture]\label{th1}
 The polynomials $W_{g,n}$ can be determined by the either of the
 following recursive formulas
 \begin{align*}
 W_{g,n}(t_1,t_{L'})=&-\res_{z=0}\Bigl(K(z,t_1)\,
     \tW_{g,n}\bigl(\frac1z,\frac1z;t_{L'}\bigr)\Bigr)\\
  =&\res_{z=0}\Bigl(K(z,t_1)\,
     \tW_{g,n}\bigl(\frac1z,\frac1{\si(z)};t_{L'}\bigr)\Bigr)\\
=&-\res_{z=0}\Bigl(K(z,t_1)\,
     \tW_{g,n}\bigl(\frac1{\si(z)},\frac1{\si(z)};t_{L'}\bigr)\Bigr)
  \end{align*}
 where
 $$K(z,t_1)=\frac{t_1^2(1+t_1)}{2(1-z\,t_1)(1-\si(z)\,t_1)}\;\frac{z\,dz}{z+1}$$
 and the series $\si(z)=-z+\frac{2}{3}z^2-\frac{4}{9}z^3+\dots$ is
 defined in the next subsection.
 \end{theorem}

The second equality is a reformulation of the Bouchard-Mari\~no
conjecture. Experiments show, however, that the first formula
is more efficient for practical computations.

Analytically, the meaning of this theorem is as follows. The
function $H_{g,n}$ is defined originally as a formal power series
expansion at $x_i=0$. It turns out, however, that this series has
a finite radius of convergence with respect to each variable $x_i$
(to be precise, the radius of convergence is $e^{-1}$). An attempt
to extend it beyond the radius of convergence meets
difficulties: the function becomes multi-valued with ramification
at $x_i=e^{-1}$. Therefore, it is more natural to
consider~$H_{g,n}$ as a function on the product
$C\times\dots\times C$ where $C$ is the curve given by the
equation $x=\Bigl(1+\frac1{t}\Bigr)\,e^{-1-\frac1{t}}$. When
treated in this way, it becomes single-valued and even rational. The
recursive relation of the theorem is formulated in terms of the
analysis of the behavior of the function $H_{g,n}$ (and closely
related to it function $W_{g,n}$) in a neighborhood of the
ramification point $x_1=e^{-1}$ which is different from the
origin.

\subsection{The Lambert curve}

The \emph{Lambert curve} is a curve in~$\C^2$ defined by the equation
\begin{equation}
x=y\,e^{-y}.
 \label{lam}\end{equation}
We consider this affine curve as an open part of its
compactification~$C=\C P^1$. We regard~$y$ as a rational
coordinate on~$C$ and the projection to the~$x$-line as a
holomorphic function with an essential singularity at the point
$y=\infty$. In addition to~$y$ we use other convenient
rational coordinates on~$C$. In particular, we keep the
notations~$z$ and~$t$ for the rational coordinates related to~$y$
by
$$y=1+z=1+\frac1t,\quad t=\frac1z.$$

There are two points on~$C$ of special interest for us: the
\emph{origin}~$O$ corresponding to the coordinates~$y=0$, $z=-1$,
$t=-1$, and the \emph{branching point}~$P$ with the
coordinates~$y=1$, $z=0$, $t=\infty$. The point~$P$ is a 
Morse critical point for the function~$x$. It means that the projection to
the $x$-line considered as a branched cover has ramification of
order two at~$P$.

Consider also the function $w=\log x$. It is multi-valued, however, its differential is a well-defined meromorphic differential on~$C$,
\begin{equation}
dw=\frac{dx}{x}=\frac{1-y}{y}dy=-\frac{z}{z+1}dz=\frac{dt}{t^2(t+1)}.
 \label{dw}\end{equation}
Denote also by~$D$ the vector field dual to this $1$-form,
\begin{equation}
D=x\d_x=\frac{y}{1-y}\d_{y}=-\frac{z+1}{z}\d_{z}
 =t^2(t+1)\d_{t}.
 \label{D}
 \end{equation}

We regard~\eqref{dw} and~\eqref{D} as a single meromorphic form and a single vector
field on~$C$ respectively, whose coordinate presentation depends
on the chosen local coordinate. Remark that the form $dw$ vanishes
at ~$P$, while the field~$D$ has a simple pole at this point.

The inversion of~\eqref{lam} near the origin is given~\cite{CGHJK96,DMSS12} by the
expansion
$$y=\sum_{\mu=1}^\infty \frac{\mu^{\mu-1}}{\mu!} x^\mu.$$
It follows from~\eqref{D} that for any integer~$k$ the series
$$\xv_k=\sum_{\mu=1}^\infty \frac{\mu^{\mu+k}}{\mu!} x^\mu=D^{k+1}y$$
is a rational function on~$C$. More explicitly, in the $t$-coordinate
it is given for~$k\ge0$ by the recursion
$$\xv_0(t)=-1-t,\qquad
 \xv_{k+1}(t)=t^2(t+1)\frac{d}{dt}(\xv_k(t)).$$
It is a polynomial in~$t$:
$$\xv_k(t)=-k!\,t^{k+1}-\dots-(2k-1)!!\,t^{2k+1}.$$
The degree of this polynomial is~$2k+1$. Equivalently, one can say
that~$\xv_k$ considered as a meromorphic function on~$C$ has pole
of order~$2k+1$ at~$P$. It follows that the linear span of the
polynomials~$\xv_k$ form a subspace of `approximately half'
dimension in the space of all polynomials in~$t$. This subspace
has a nice characterization that we describe now.

Denote by~$\si$ the involution interchanging the sheets of the
ramification defined by the function~$x$ near the point~$P$. The
function $\si$ is holomorphic in a neighborhood of~$P$ and its
Taylor expansion can be computed from the equation
\begin{equation}
(1+z)\,e^{-z}=(1+\si(z))\,e^{-\si(z)}.
 \label{s(z)}\end{equation}
Here are the first few terms of this expansion written in the
coordinates~$z$ and~$t$, respectively:
\begin{align*}
 \si(z)&=-z+\frac{2}{3}z^2-\frac{4}{9}z^3
   +\frac{44}{135}z^4-\frac{104}{405}x^5+\frac{40}{189}z^6+\dots\\
 \Si(t)&=\frac1{\si(1/t)}=-t-\frac{2}{3}-\frac{4}{135t^2}
   +\frac{8}{405t^3}-\frac{8}{567t^4}+\dots
\end{align*}

\begin{lemma}\label{lem2}
For any $k\ge0$ the principal part of the pole of $\xv_k(t)$ at
the point~$P$ is odd with respect to the involution~$\si$. In other
words, the function $\xv_k(t)+\xv_k(\Si(t))$ is holomorphic at~$P$.
\end{lemma}

\begin{proof} For $k=0$ the assertion is obvious since the
principal part of any simple pole is odd. Now, arguing by
induction, we assume that $\xv_k$ is represented in the form
$$\xv_k(t)=\eta_k(t)+F_k(t)$$
where $\eta_k(t)=\frac12\bigl(\xv_k(t)-\xv_k(\Si(t))\bigr)$ is odd
and $F_k(t)=\frac12\bigl(\xv_k(t)+\xv_k(\Si(t))\bigr)$ is even and
holomorphic at~$P$. Then, by definition,
$$\xv_{k+1}(t)=D(\xv_k(t))=D(\eta_k(t))+D(F_k(t)).$$
The field~$D$ is invariant with respect to the involution,
therefore, it preserves the parity. It follows that~$D(\eta_k(t))$
is odd, and $D(F_k(t))$ is even and the order of its pole at~$P$
is at most~$1$. It follows that $D(F_k(t))$ is, in fact,
holomorphic at~$P$, which proves the lemma.
\end{proof}

\subsection{The cut-and-join equation}

Yet another way to collect Hurwitz numbers into a generating
series is given by the expansion
$$
 G_{g,n}(p_1,p_2,\dots)=\frac1{n!}\sum_{\mu_1,\dots,\mu_n}
 \frac{\ch_{g;\mu_1,\dots,\mu_n}}{\bn(g,\mu)!}
     \;p_{\mu_1}\dots p_{\mu_n}.
$$

The series $G_{g,n}$ involves an infinite collection of variables
$p_1,p_2,\dots$ and all its terms are homogeneous of degree~$n$.
The relation between the two series $G_{g,n}$ and $\cH_{g,n}$ is
obvious. In particular, $G_{g,n}$ can be obtained from
$\frac1{n!}\cH_{g,n}$ by replacing every monomial $x_1^{\mu_1}\dots
x_n^{\mu_n}$ by the corresponding monomial $p_{\mu_1}\dots
p_{\mu_n}$.

The cut-and-join equation is a recursion on Hurwitz numbers obtained through the analysis of the cyclic type of the result of 
multiplication of a given permutation by a single transposition.
In its original form~\cite{GJ}, it is written as
$$2u\pd{e^G}{u}+\sum_{i=1}^\infty (i+1)\,p_i\pd{e^G}{p_i}=
  \frac12\sum_{a,b}\Bigl((a+b)p_ap_b\pd{e^G}{p_{a+b}}
   +u\,ab\,p_{a+b}\pd{^2e^G}{p_a\D p_b}\Bigr)$$
where
$$G=\sum_{g,n}u^{g-1}G_{g,n}.$$
The same equation written in terms of the individual components
$G_{g,n}$ is
\begin{align}
\label{cutjoinp}
& (2g-2+n)G_{g,n}+\sum_{i=1}^\infty
 i\,p_i\pd{G_{g,n}}{p_i}
 \\ \notag &
 =
  \frac12\sum_{a,b}\Bigl((a+b)p_ap_b\pd{G_{g,n-1}}{p_{a+b}}
  \\ \notag &
  \phantom{ = }\ 
  +ab\,p_{a+b}\Bigl(\pd{^2G_{g-1,n+1}}{p_a\D p_b}
   +\sum_{\substack{g_1+g_2=g\\n_1+n_2=n+1}}
   \pd{G_{g_1,n_1}}{p_a}\pd{G_{g_2,n_2}}{p_b}\Bigr)\Bigr).
  \end{align}

Let us rewrite this equation in terms of the functions $\cH_{g,n}$.
The operator $\sum i\,p_i\d_{p_i}$ from the left hand side of the
equation corresponds to the operator $\sum_{i=1}^n D_i$ acting
on $\cH_{g,n}$ where
$$D_i=x_i\d_{x_i}=t_i^2(t_i+1)\d_{t_i}.$$
The action of the `cut' operator $\sum(a+b)p_ap_b\D_{p_{a+b}}$ in
terms of the series $\cH_{g,n}$ results in the replacement of any
monomial $x_m^\ell$ by the sum
\begin{align*}
\sum_{a+b=\ell}(a+b)x_j^ax_k^b
 & =\ell\frac{x_kx_j(x_k^{\ell-1}-x_j^{\ell-1})}{x_k-x_j}
\\
 & =\frac{x_j}{x_k-x_j}x_k\pd{(x_k^\ell)}{x_k}
    +\frac{x_k}{x_j-x_k}x_j\pd{(x_j^\ell)}{x_j}
\\ 
 & =\frac{x_j}{x_k-x_j}D_k(x_k^\ell)+\frac{x_k}{x_j-x_k}D_j(x_j^\ell).
\end{align*}

In a similar way, the action of the `join' operator
$ab\,p_{a+b}\pd{^2}{p_a\D p_b}$ results in the replacement of any
monomial $x_j^a x_k^b$ by the monomial
$$ ab\,x_m^{a+b}=\Bigl(x_m\pd{(x_m^a)}{x_m}\Bigr)
 \Bigl(x_m\pd{(x_m^b)}{x_m}\Bigr)
 =D_m(x_m^a)D_m(x_m^b).$$
The relation between the indices $k,j$, and $m$ in the above
considerations is not essential. One should only take care that
the result is symmetric with respect to the permutations of the
variables $x_1,\dots,x_n$.

The relation obtained from~\eqref{cutjoinp} in this way is
presented below. In this relation $L$ denotes the collection of
indices $L=\{1,2,\dots,n\}$, and $t_L=(t_1,\dots,t_n)$.
\begin{align}\label{cutjoint1}
& (2g{-}2{+}n) \cH_{g,n}(t_L)+\sum_{k=1}^nD_k\cH_{g,n}(t_L)
\\ \notag &
=
 \frac12\sum_{k\ne j}2
  \frac{x_j}{x_k-x_j}D_k\cH_{g,n-1}(t_{L\backslash\{j\}})
\\ \notag &
\phantom{ = }\ +\frac12\sum_{k=1}^n\Bigl(D_kD_{n+1}\cH_{g-1,n+1}(t_L,t_{n+1})\bigm|_{t_{n+1}=t_k}
\\ \notag &
\phantom{ = }\ +\sum_{g_1+g_2=g}\sum_{A\sqcup B=L\setminus\{k\}}D_k\cH_{g_1,|A|+1}(t_k,t_A)
    D_k\cH_{g_2,|B|+1}(t_k,t_B)\Bigr),
\end{align}
where the last summation is taken over the set of all possible
partitions of the index set
$L\setminus\{k\}=\{1,\dots,k-1,k+1,\dots,n\}$ into a disjoint
union of two subsets, $A$ and~$B$.

This relation can be regarded as a relation on the functions in
either $x$ or $t$-variables, where $x_i$ and $t_i$ are related
by~\eqref{xt}. We consider this relation as the `preliminary form'
of the required cut-and-join equation. The final form is obtained
by extracting unstable terms from the last summation corresponding
to the functions $\cH_{0,1}$ and $\cH_{0,2}$ and combining these terms
with the corresponding terms of the previous sums.
Using~\eqref{H0102}, we find the coefficients of the recombined
terms
\begin{align*}
1-D_1\cH_{0,1}(t_1)&=-\frac1{t_1},\\
\frac{x_2}{x_1-x_2}+D_1\cH_{0,2}(t_1,t_2)&=\frac{t_1^2(1+t_2)}{t_1-t_2} .
\end{align*}

We obtain thus the final form of the cut-and-join equation in the
$t$-coordinates, see more details in~\cite{MZ}:
\begin{align}\label{cutjoint2}
& 
(2g{-}2{+}n)\cH_{g,n}(t_L)+\sum_{k=1}^n\Bigl(-\frac{1}{t_k}\Bigr)D_k\cH_{g,n}(t_L)
\\ \notag &
= \sum_{k\ne j}
  \frac{t_k^2(1+t_j)}{t_k-t_j}D_k\cH_{g,n-1}(t_{L\backslash\{j\}})
\\ \notag & \phantom{ = }\
 +\frac12\sum_{k=1}^n\Bigl(D_kD_{n+1}\cH_{g-1,n+1}(t_L,t_{n+1})\bigm|_{t_{n+1}=t_k}
\\ \notag & \phantom{ = }\
 +\sum_{g_1+g_2=g}\sum_{A\sqcup B=L\setminus\{k\}}^{\rm stable}
   D_k\cH_{g_1,|A|+1}(t_k,t_A)D_k\cH_{g_2,|B|+1}(t_k,t_B)\Bigr) .
\end{align}

It is remarkable that the `non-polynomial' summands are canceled
out, and both sides of the relation proved to be polynomial in
$t$-variables. As it is pointed out in~\cite{MZ}, selecting the
highest and the lowest degree terms of this formula one gets
immediately the Virasoro constrains for the intersection numbers
of $\psi$-classes on the moduli spaces of curves (see e.g. \cite{D,K,W,MMM92,KMMMZ}) and
the relation of the $\l_g$-formula~\cite{FP1,FP2}, respectively.

\subsection{Reduction by symmetrization}

The cut-and-join equation~\eqref{cutjoint2} can be
used to determine $\cH_{g,n}$ inductively. However, in the presented
form it is not very convenient since it is not clear how to invert
the operator on the left hand side of the equation. It is not even
obvious that the function~$\cH_{g,n}$ obtained by this recursion is
polynomial in $t$-variables. The following two key observations
of~\cite{EMS09} lead to a considerable simplification
of~\eqref{cutjoint2}: 
\begin{enumerate}
\item The function $\cH_{g,n}$ is polynomial in
each variable~$t_i$, therefore, the whole information about this
function is contained in the principal part of its pole at the
point~$P$ with respect to~$t_i$. Let us stress that in \cite{EMS09} this polynomiality was derived from the ELSV formula, while in the present paper we have it independently due to the results of Section \ref{sec:polynomiality}, as noted above.
\item The principal part of the
pole of~$H_{g,n}$ is odd with respect to the involution~$\si$ on
each $t_i$-line (as it follows from Lemma~\ref{lem2}).
\end{enumerate}

Consider the \emph{even} summand of the principal part of the pole
at~$P$ of each term in~\eqref{cutjoint2} \emph{with respect to the
first variable~$t_1$}. It follows that most of the terms will give
trivial contribution to the result so that the whole equation will
be considerably simplified.

It is more convenient for us to use a slight modification of this
idea. Namely, set
$$\eta(t_1)=\si\Bigl(\frac1{t_1}\Bigr)-\frac1{t_1}.$$
This function is holomorphic at~$P$ and odd with respect to the
involution. Now, for any meromorphic function $f(t_1)$ we denote
by
$$\odd{\frac{f(t_1)}{\eta(t_1)}}$$
the \emph{odd residueless principal part} of the pole of the
quotient $f/\eta$ at the point~$P$. More explicitly, if we write
the Laurent expansion
$$\frac{f(t_1)+f(\Si(t_1))}{2\,\eta(t_1)}=
   \sum_{-\infty<i\le N} a_i\,t_1^{i}$$
at~$P$, then we set, by definition,
$$\odd{f/\eta}= \sum_{i=2}^{N} a_i\,t_1^{i}.$$
From this definition we see that $\odd{f/\eta}$ is
a polynomial in~$t_1$ divisible by~$t_1^2$.

We apply the transformation $f(t_1)\mapsto \odd{2f/\eta}$ to both
sides of~\eqref{cutjoint2}. This transformation annihilates any
function in~$t_1$ whose pole at $P$ has odd principal part. In
particular, it annihilates $\cH_{g,n}(t_L)$ on the left hand side as
well as all terms on both sides of the equality corresponding to
the summation index $k$ different from~$1$.

Let us compute the action of this transformation on the therm
$-\frac{1}{t_1}D_1\cH_{g,n}$ on the left hand side. For any
meromorphic function $f(t_1)$ which is odd with respect to the
involution we have
$$\frac{\frac{-f(t_1)}{t_1}
 -\frac{f(\Si(t_1))}{\Si(t_1)}}{\eta(t_1)}
 =f(t_1)\,\frac{-\frac{1}{t_1}+\frac{1}{\Si(t_1)}}{\eta(t_1)}
  =f(t_1).$$
Therefore, $\odd{-\frac{2 f(t_1)}{\eta(t_1)\,t_1}}=\odd{f(t_1)}$.
The function $D_1\cH_{g,n}$ differs from such a function by a
holomorphic summand that gives trivial contribution to the
transformation. This implies
$$\odd{-\frac{2}{\eta(t_1)\,t_1}D_1\cH_{g,n}}
  =\odd{D_1\cH_{g,n}}=D_1\cH_{g,n}.$$

We obtain finally the equation
\begin{equation}
D_1\cH_{g,n}=\odd{\frac1\eta\left(
\begin{aligned}
 &\sum_{j=2}^n
  \frac{t_1^2(1+t_j)}{t_1-t_j}D_1\cH_{g,n-1}(t_1,t_{L'\backslash\{j\}})\\
 &+D_1D_{n+1}\cH_{g-1,n+1}(t_1,t_{L'},t_{n+1})\bigm|_{t_{n+1}=t_1}\\
 &+\sum_{g_1+g_2=g}
  \sum_{A\sqcup B=L'}^{\rm stable}D_1\cH_{g_1,|A|+1}(t_1,t_A)
    D_1\cH_{g_2,|B|+1}(t_1,t_B)
\end{aligned}
\right)}
 \label{Ht}\end{equation}
where $L'=L\setminus\{1\}=\{2,\dots,n\}$.

In order to represent this equation in a more readable form, let
us apply $\prod_{k=2}^nD_k$ to both its sides and observe that the
expression inside the square brackets becomes algebraic with
respect to the functions $W_{g,k}$ defined by \eqref{Wset}.
Moreover, the first term on the right hand side can be formally
included into the last since we defined the contribution
of the unstable terms as in  Equations~\eqref{wunst1} and~\eqref{wunst2}.
With this notation, the result of the application
of~$\prod_{k=2}^nD_k$ to both sides of Equation~\eqref{Ht} takes the form
of the following recursive relation on $W_{g,n}$.

\begin{proposition}\label{th3}
The function $W_{g,n}$ defined by~\eqref{Wset}--\eqref{wunst2}
satisfies the recursive equation
$$W_{g,n}(t_1,t_{L'})=\odd{\frac1{\eta(t_1)}\tW(t_1,t_1;t_{L'})},$$
where $L'=\{2,\dots,n\}$, $t_{L'}=(t_2,\dots,t_n)$, and
\begin{align}  \label{tW}
 & \tW_{g,n}(u,v;t_{L'})=W_{g-1,n+1}(u,v,t_{L'})
 \\ \notag
 & +\sum_{g_1+g_2=g}
  \sum_{A\sqcup B=L'}W_{g_1,|A|+1}(u,t_A)W_{g_2,|B|+1}(v,t_B).
\end{align}
\end{proposition}

\begin{remark}\label{rem4} If $f(t_1)$ is a meromorphic function whose pole at~$P$
has odd principal part then for any other function $g$ we have
$$\odd{\frac{f(t_1)g(t_1)}{\eta(t_1)}}+
  \odd{\frac{f(t_1)g(\Si(t_1))}{\eta(t_1)}}=
 \odd{f(t_1)\frac{g(t_1)+g(\Si(t_1))}{\eta(t_1)}}=0$$
since $(g(t_1)+g(\Si(t_1)))/\eta(t_1)$ is odd. Therefore, $W_{g,n}$
can equivalently be obtained by the either of the following
relations
\begin{align*}
W_{g,n}(t_1,t_{L'}) & =-\odd{\frac1{\eta(t_1)}\tW_{g,n}(t_1,\Si(t_1);t_{L'})}
\\ 
   &  =\odd{\frac1{\eta(t_1)}\tW_{g,n}(\Si(t_1),\Si(t_1);t_{L'})}.
\end{align*}
\end{remark}

\subsection{Residual formalism}

The coefficient $f_k$ of the meromorphic function
$f(t_1)=\sum_{-\infty<i\le N} f_i t_1^i$ can be extracted by
taking the residue
  $$f_k=\Res_{z=0}\Bigl(f\bigl(\tfrac1z\bigr)z^{k-1}\,dz\Bigr).$$
It follows that the whole residueless principal part of the pole
of~$f$ is given by
\begin{equation}
  \sum_{k=2}^N f_kt_1^k
  =\Res_{z=0}\Bigl(f\bigl(\tfrac1z\bigr)\sum_{k=2}^\infty t_1^kz^{k-1}\,dz\Bigr)
 =\Res_{z=0}\Bigl(f\bigl(\tfrac1z\bigr)\frac{t_1^2 z}{1-t_1 z}\,dz\Bigr).
 \label{res1}\end{equation}
Similarly, for the function $\of(t_1)=f(\Si(t_1))=\sum_{-\infty<i\le
N} \of_i t_1^i$ we get
\begin{align} \label{res2}
\sum_{k=2}^N \of_kt_1^k
& =\Res_{z=0}\Bigl(f\bigl(\tfrac1{s(z)}\bigr)\frac{t_1^2 z}{1-t_1 z}\,dz\Bigr)
\\ \notag &
 =\Res_{z=0}\Bigl(f\bigl(\tfrac1{z}\bigr)\frac{t_1^2 \si(z)}{1-t_1 \si(z)}
  \frac{z}{1+z}\frac{1+\si(z)}{\si(z)}\,dz\Bigr).
\end{align}
  We used here the equality
  $$\frac{z\,dz}{1+z}=\frac{\si(z)\,d\si(z)}{1+\si(z)}$$
that follows from Equation~\eqref{s(z)}.

Combining~\eqref{res1} and~\eqref{res2} we obtain a residual
formula for the odd residueless principal part of the pole of a
function:
$$\odd{f(t_1)/\eta(t_1)}=-\Res_{z=0}\bigl(K(z,t_1)\,f(1/z)\bigr)$$
where
\begin{align}  \label{K}
K(z,t_1)&=\frac{1}{2\eta(1/z)}\Bigl(
 \frac{t_1^2z}{1-t_1 z}-\frac{t_1^2\si(z)}{1-t_1\si(z)}\frac{z}{1+z}
   \frac{1+\si(z)}{\si(z)}\Bigr)\,dz \\ \notag
   &=\frac{t_1^2(1+t_1)}{2(1-z\,t_1)(1-\si(z)\,t_1)}\;\frac{z\,dz}{z+1}.
\end{align}
 This, substituted into the recursive formulas
 of Proposition~\ref{th3} and Remark~\ref{rem4}, directly gives Theorem~\ref{th1}.



\section{Spectral curve topological recursion/Givental correspondence revisited} \label{sec:CEO-Givental}
In this section we review the correspondence between spectral curve topological recursion and Givental theory established in \cite{DOSS12}. 
We use it in the next section to prove the equivalence between the Bouchard-Mari\~no conjecture and the ELSV formula. 
This way we obtain a new proof of the ELSV formula, using the new independent proof of the Bouchard-Mari\~no conjecture from the previous section.

\subsection{Givental formula}

Let $H$ be a Frobenius algebra, that is, a finite-dimensional commutative associative algebra over~$\C$ with a unit denoted by~$\1\in H$, equipped with a linear function $\ell:H\to\C$ such that the symmetric bilinear form given by $\langle a,b\rangle=\ell(a\,b)$ is non-degenerate. A typical example is the (even part of the) cohomology ring of a complex compact manifold. Its dimension will be denoted by~$N=\dim H$. Fix a basis $e_1,\dots,e_N$ in~$H$.

Consider also an element of the \emph{Givental upper triangular twisted loop group}, that is, a formal series of the form
$$R(z)=1+\sum_{k=1}^\infty R_kz^k,\quad R_k\in\End(H),$$
satisfying
$$R(z)\,R^*(-z)=1.$$
In terms of the Lie algebra element $r(z)=\log(R(z))$, $R(z)=\exp{r(z)}$, the last relation can be equivalently rewritten as $r(z)+r^*(-z)=0$.

To this data (a Frobenius algebra and an element $R$ of the upper triangular group) Givental associates a \emph{formal Gromov-Witten potential} $F$, a formal series in an infinite number of variables $t_{k\nu}$, $k=0,1,2,\dots$, $\nu=1,2,\dots,N$, and one extra variable $\hbar$, defined by the formula
\begin{equation}\label{Gact}
e^{\frac1\hbar F}=\widehat R\,e^{\frac{1}{\hbar} F^\rmtop},
  \quad\widehat R=e^{\widehat r},
\end{equation}
where $F^\rmtop$ is the potential of the topological field theory associated with the Frobenius algebra~$H$, and $\widehat r$ is a second-order differential operator obtained from $r(z)$ by a procedure of `quantization of quadratic Hamiltonians', see details in~\cite{G}.

A choice of basis in~$H$ is not essential. A change of the basis leads to a linear change of variables in the potential of the form $t_{k\nu}\longrightarrow \sum_{\mu=1}^N\Psi^\mu_\nu t_{k\mu}$ where $\Psi$ is the matrix of the change of basis. 
In other words, we can treat $F$ as a formal function on $H\otimes H\otimes\dots$.

It was observed in~\cite{FSZ,Kaz07} that the potential $F$ constructed this way is, in fact, a descendant potential of a certain 
cohomological field theory. Moreover, it is proved in~\cite{T} that the descendant potential of any semi-simple cohomological field theory can be represented in such form.

\subsection{Spectral curve topological recursion}

Spectral curve topological recursion is a formal procedure leading to a family of certain differentials $w_{g,n}$ associated with a plane complex curve. They were introduced originally for particular curves in relation to matrix models in mathematical physics \cite{AMM06a, AMM06,CE05, CE06, CEO},  then the procedure was formalized for arbitrary abstract curves \cite{ EO, OrT}.

Let $C\subset\C^2$ be a smooth complex curve on the plane with coordinates $x,y$. Let $a_1,\dots,a_N\in C$ be the critical points of the coordinate function~$x$. The construction of the differentials $w_{g,n}$ requires the study of the curve in a neighborhood of these points, therefore, it is sufficient to assume that instead of~$C$ we have a union of~$N$ small discs centered at the points~$a_i$, $i=1,\dots,N$, or even the union of formal neighborhoods of these points. Respectively, by a function or differential form (holomorphic or meromorphic) on~$C$ we mean a collection of germs of functions or differential forms at the points~$a_i$ or even a collection of formal Laurent series at these points.

Assume that each point~$a_i$ is a Morse critical point of the function~$x$, that is, $x$ is a ramified covering with a ramification of order~$2$ at~$a_i$. Let~$\si$ be the holomorphic involution on~$C$ interchanging the branches of the function~$x$ near~$a_i$. In order to simplify notations, for any function or differential form $\a$ we denote $\overline\a=\si^*\a$. With this notation the involution is given by $\s\colon (x,y)\mapsto(x,\overline y)$. Remark that this bar sign has nothing to do with the complex conjugation in the present context. Remark also that the form $\overline\a$ is defined in a neighborhood of the point~$a_i$ only, even if the form~$\a$ is globally defined.

On top of that, assume that we are given a \emph{$2$-point differential} $B(z_1,z_2)$ (referred to as \emph{Bergman kernel} in some papers), that is, a meromorphic symmetric $2$-differential on~$C\times C$ representable near $a_i\times a_j\in C\times C$ in the form
$$B(z_1,z_2)=\delta_{i,j}\frac{dz_1^{(i)}dz_2^{(j)}}{(z_1^{(i)}-z_2^{(j)})^2}
 +B^{(ij)}_{\reg}(z_1^{(i)},z_2^{(j)})$$
where $z^{(i)}$ is a local coordinate on~$C$ near $a_i$ and where~$B^{(ij)}_{\reg}(z_1^{(i)},z_2^{(j)})$ is holomorphic at $a_i\times a_j$.

The spectral curve $n$-point functions $w_{g,n}$, $g\ge0$, $n\ge 1$, are meromorphic $n$-differentials on~$C^{\times n}$ defined inductively by the following formulas:
$$w_{0,1}(z)=0,\qquad w_{0,2}(z_1,z_2)=B(z_1,z_2),$$
and for $2g-2+n>0$,
\begin{equation}\label{eqB1}
w_{g,n}(z,z_2,\dots,z_n)=-\sum_{i=1}^N\res_{z'=a_i}\left(
\frac{\widetilde w_{g,n}(z',\bar z',z_2,\dots,z_n)}{2\,\mu(z')}
 \int_{z'}^{\bar z'} B(z,\cdot)\right),
 \end{equation}
where $\mu$ is the $1$-form $\mu:=y\,dx-\bar y\,dx$ defined in a neighborhood of the union of points~$a_i$, and where
\begin{equation}\label{eqEO}
\widetilde w_{g,n}(z',z'',z_K)=
  w_{g-1,n+1}(z',z'',z_K)+
\sum_{\substack{g_1+g_2=g\\I\sqcup J=K}}\!\!
 w_{g_1,|I|+1}(z',z_I)\;w_{g_2,|J|+1}(z'',z_J).
\end{equation}
We used here notation $K=\{2,\dots,n\}$, and  $u_I=(u_{i_1},\dots,u_{i_{|I|}})$ for any subset $I=\{i_1,\dots,i_{|I|}\}\subset K$.

\begin{remark}
We collect here several important remarks clarifying the meaning of all these formulas.

\begin{enumerate}
\item Consider the following operator $\a\mapsto P\a$ acting on the space of meromorphic $1$-forms,
$$(P\a)(z)=
\sum_{i=1}^N\res_{z'=a_i}\left(
\frac{\a(z')}{2}
 \int_{z'}^{\bar z'} B(z,\cdot)\right).$$
Denote by $L$ the image of this operator. Then \emph{the operator~$P$ is the projection to the subspace~$L$}, that is, it is identical on~$L$. The kernel of~$P$ is generated by holomorphic and by even (in the sense of local automorphism $\sigma$) meromorphic $1$-forms.

\item It follows that the $1$-form in~$z$ on the right hand side of~\eqref{eqB1} belongs to~$L$. In other words, the invariants $w_{g,n}$ can be regarded as tensors $w_{g,n}\in L^{\otimes n}$ (for $(g,n)\ne(0,2)$). These tensors are \emph{symmetric} and \emph{polynomial}. The last property means that $w_{g,n}$ belongs to the corresponding tensor product space itself, not just to its completion.

\item The data contained in the collection of invariants $w_{g,n}$ can be collected in a single \emph{potential} $F=\sum \hbar^g F_g$ such that the symmetric tensor $w_{g,n}$ is identified with the $n$th homogeneous term of the Taylor expansion of $F_g$,
$$w_{g,n}=\sum_{\a_1,\dots,\a_n}\pd{^n F_g}{t_{\a_1}\dots \d t_{\a_n}}\Bigm|_{t=0} d\xi_{\a_1}\otimes\dots\otimes d\xi_{\a_n}.$$
Here $\{d\xi_a\}_{\a\in\cA}$ is some chosen basis in~$L$, and $t=\{t_\a\}_{\a\in\cA}$ is the set of formal variables labeled by the same set of indices. The coordinate expression of the potential $F$ depends on a choice of the basis in~$L$. A different choice of the basis leads to the corresponding linear change of coordinates in~$F$. Otherwise, $F$ can be regarded as a formal function on the infinite dimensional space~$L^*$; with this treatment of the potential, it is invariantly defined and independent of any basis.

\item
The dual space~$V=L^*$ can be identified with the space of \emph{odd holomorphic}~$1$-forms. The pairing is given by
$$(\a,\b)=\sum_{\nu=1}^N\res_{z=a_\nu}(\a\;\int\!\!\b),\quad\a\in L.\quad\b\in V.$$
If $\{d\xi_\a\}_{\a\in\cA}$ is any basis in~$L$ and $\{d\xi^\a\}_{\a\in\cA}$ is the dual basis in~$V=L^*$, then there is an asymptotic expansion
$$\frac12(B(z,w)-B(z,\overline w))=\sum_{\a\in\cA}d\xi_\a(z)d\xi^\a(w).$$
This expansion takes place as $w\to a_i$, $|w-w(a_i)|\ll|z-z(a_i)|$.

\item
It follows, in particular, that the subspace $L$ is spanned by the coefficients of the Taylor expansion of the antisymmetrized Bergman kernel $\frac12(B(z,w)-B(z,\overline w))$ with respect to the second argument~$w$ at the points $a_i$.
\end{enumerate}

\end{remark}


\subsection{Givental action as spectral curve topological recursion}

Here we formulate in a refined way the result of~\cite{DOSS12} in the case $N=1$.

Let $C$ be a curve on the $(x,y)$-plane as above. Consider the following operator acting in the space of meromorphic $1$-forms,
$$\Do:\a\mapsto d\Bigl(\frac\a{dx}\Bigr).$$
This operator commutes with the action of the involution~$\s$, $\Do\overline\a=\overline{\Do\a}$. Set
$$
d\xi^k:=\Do^{-k}dy,\quad k=0,1,2,\dots.
$$
The forms $d\xi^k$ are holomorphic in a neighborhood of the point~$a_1$. There is an ambiguity in the choice of integration constants appearing in the inversion of~$D$. Different choices of these constants lead to forms that differ by a holomorphic and \emph{even} (with respect to the involution $\s$) summand. It follows that the odd parts of these forms
$$\frac12\bigl(d\xi^k-d\overline\xi^k\bigr),\quad k=0,1,2,\dots$$
are independent of any choice. Moreover, these odd forms form a basis in the space of odd holomorphic forms. Let us take the antisymmetrized Bergman kernel $\frac12\bigl(B(z,w)-B(z,\overline w)\bigr)$, develop it over the obtained basis, and denote by $d\xi_k$ the coefficients of this expansion:
\begin{equation}\label{Bxiexp}
\frac12\bigl(B(z,w)-B(z,\overline w)\bigr)
 =\sum_{k=0}^\infty d\xi_k(z)\frac{d\xi^k(w)-d\overline\xi^k(w)}2.
 \end{equation}
This asymptotic expansion takes place as $w\to 0$, $|w|\ll|z|$ where $z$ is a local holomorphic coordinate on~$C$ near the point~$a_1$. The form $d\xi_k$ defined by this expansion is meromorphic with a pole of order $2k+1$ at $z=0$.

\begin{definition}
The Bergman kernel is said to be \emph{compatible} with the operator~$\Do$ if the introduced meromorphic forms $d\xi_k$ are given explicitly by $d\xi_k=(-1)^{k+1}\Do^{k+1}d\xi^0$.
\end{definition}

The following criterion simplifies the verification of the compatibility condition.

\begin{lemma}
\label{lcomp}
Assume that the Bergman kernel satisfies the identity
$$(\Do_z+\Do_w)B(z,w)=-\Do_zd\xi^0(z)\;\Do_wd\xi^0(w).$$
Then it is compatible with~$\Do$.
\end{lemma}

\begin{proof} Applying the expansion~\eqref{Bxiexp} we get
\begin{align*}
0 & =(\Do_z+\Do_w)\frac{B(z,w)-B(z,\overline w)}2+
 \Do_zd\xi^0(z)\;\Do_w\frac{d\xi^0(w)-d\overline\xi^0(w)}2
 \\ &
 =\sum_{k=0}^\infty (\Do_zd\xi_k(z)+d\xi_{k+1}(z))\frac{d\xi^k(w)-d\overline\xi^k(w)}2
 \\ & \phantom{ = }\ 
 +(\Do_zd\xi^0(z)+d\xi_0(z))\,\Do_w\frac{d\xi^0(w)-d\overline\xi^0(w)}2.
\end{align*}
This equality is equivalent to the system of equations $d\xi_0=-\Do d\xi^0$, $d\xi_{k+1}=-\Do d\xi_k$, that is, $d\xi_k=(-1)^{k+1}\Do^{k+1}d\xi^0$, as required.
\end{proof}

Now, assume that the Bergman kernel is compatible with~$\Do$. Introduce the local coordinate~$s$ on the curve near the point~$a_1$ from the relation $dx=s\,ds$, that is,
$$s=\sqrt{2(x-x(a_1))}.$$
This coordinate is defined up to a sign, and the involution in this coordinate is given simply by $\overline s=-s$. Consider the expansion of the odd part of the form $dy$ in this coordinate,
\begin{equation}
\label{ys}
\frac12(dy-d\overline y)=ds+\sum_{k=1}^\infty R_k\frac{s^{2k}\,ds}{(2k-1)!!}.
\end{equation}

We can now formulate the main result of~\cite{DOSS12} for the case of $N=1$.

\begin{theorem}
\label{DBOSSth}
If the Bergman kernel is compatible with the operator~$\Do$, then the spectral curve $n$-point functions are the $n$-point correlator functions of a certain formal GW potential $F(t_0,t_1,\dots)=\sum \hbar^gF_g$,
$$w_{g,n}=\sum_{k_1,\dots,k_n}\pd{^n F_g}{t_{k_1}\dots \D t_{k_n}}\Bigm|_{t=0} d\xi_{k_1}\otimes\dots\otimes d\xi_{k_n}.$$
Moreover, this GW potential is given by the Givental formula~\eqref{Gact} with the Witten-Kontsevich potential for $F^{\rmtop}$ and with the element $R(z)=1+R_1z+R_2z^2+\dots$ of the upper triangular group whose components $R_k$ are determined by the expansion~\eqref{ys}.
\end{theorem}



\section{New proof of the ELSV formula}\label{sec:BM-ELSV}
In the present section we prove the equivalence of the Bouchard-Mari\~no formula and the ELSV formula with the help of the Givental-topological recursion correspondence reviewed in the previous section. Note that this equivalence was already proved by Eynard in \cite{Eyn11}, see also~\cite{SSZ13}. 

From this equivalence, using our new proof of the Bouchard-Mari\~no conjecture (Theorem \ref{th1}), we obtain a new proof of the ELSV formula.

\subsection{Hodge class}
The total Hodge class $\L_g=1-\l_1+\dots+(-1)^g\l_g\in H^*(\mathcal{M}_{g,n})$ provides the simplest non-trivial example of a cohomological field theory (of dimension $N=1$). It follows that its potential, the generating function for Hodge integrals,
$$
F(\hbar,t_0,t_1,\dots)=\sum_{g,n}\frac{\hbar^g}{n!}\sum_{k_1,\dots,k_n}\;
\int\limits_{\mathcal{M}_{g,n}}\!\!\!\L_g\,\psi_1^{k_1}\dots\psi_n^{k_n}~
 t_{k_1}\dots t_{k_n}
$$
is a formal GW potential. Indeed, Mumford's formula~\cite{Mu} for the Chern characters of the Hodge bundle rewritten in terms of intersection numbers has exactly the form~\eqref{Gact} with the Witten-Kontsevich potential for the series $F^\rmtop$ and the following element of the upper triangular group
\begin{equation}\label{RHodge}
R(z)=\exp\left({\sum_{n=1}^\infty\frac{B_{2n}}{2n\,(2n-1)}z^{2n-1}}\right)
 =1+\frac{1}{12}z+\frac{1}{288}z^2-\frac{139}{51840}z^3+\dots,
\end{equation}
where $B_{n}$ is the $n$th Bernoulli number. The operator $\widehat R=\exp\left({\sum_{n=1}^\infty\frac{B_{2k}}{2n\,(2n-1)}\widehat{z^{2n-1}}}\right)$ corresponding to this element acts by
$$\widehat{z^{2n-1}}=-\pd{}{t_{2n}}+\sum_{i=0}^\infty t_i\pd{}{t_{i+2n-1}}
 -\frac12\sum_{i+j=2n-2}(-1)^i\pd{^2}{t_i\D t_j}.
$$
In the definition of this operator, we use a convention which differs by the sign from that of the original paper~\cite{G}.

\subsection{BM-ELSV equivalence}
Consider the Lambert curve \eqref{lam}
$$\lx=\ly-\log(1+\ly),\qquad d\lx=\frac{\ly\,d\ly}{1+\ly},$$
which is given here in logarithmic coordinates
\begin{align*}
\lx &= -1-\log x ,\\
\ly &= -1+y .
\end{align*}

For this curve, the standard Bergman kernel $B(\ly_1,\ly_2)=\frac{d\ly_1d\ly_2}{(\ly_1-\ly_2)^2}$ is compatible with the operator $\Do$. Indeed, we have
\begin{align*}
(\Do_{\ly_1}+\Do_{\ly_2})\frac{d\ly_1d\ly_2}{(\ly_1-\ly_2)^2}& =
 d_{\ly_1}\frac{(1+\ly_1)\,d\ly_2}{\ly_1(\ly_1-\ly_2)^2}+
 d_{\ly_2}\frac{(1+\ly_2)\,d\ly_1}{\ly_2(\ly_1-\ly_2)^2} \\ &
 =-\frac{d\ly_1d\ly_2}{\ly_1^2\ly_2^2} \\ &
 =-\Do_{\ly_1}d\ly_1\;\Do_{\ly_2}d\ly_2.
 \end{align*}
Therefore, by Lemma \ref{lcomp} and Theorem \ref{DBOSSth}, the spectral curve $n$-point functions in this case are the correlation functions of a certain formal GW potential. Moreover, this potential is obtained from the Kontsevich-Witten potential by the action of the element $R(z)=1+\sum R_kz^k$ of the Givental group whose coefficients are determined by the expansion
$$\frac{d}{ds}\frac{\ly(s)-\ly(-s)}2=
1+\sum_{k=1}^\infty R_k\frac{s^{2k}}{(2k-1)!!},$$
where the function $\ly(s)$ is given by the implicit equation
$$s=\sqrt{2\,(\ly-\log(1+\ly))}.$$
It is proved in~\cite{BMex} that these coefficients are the same as those given by the expansion~\eqref{RHodge}.

This means that for our spectral curve we have
\begin{align}
\label{equivf}
w_{g,n}&=\sum_{k_1,\dots,k_n}\pd{^n F_g}{t_{k_1}\dots \D t_{k_n}}\Bigm|_{t=0} (d\xi_{k_1})_1\dots (d\xi_{k_n})_n \\ \nonumber
&=\sum_{k_1,\dots,k_n}
\br{\int_{\overline{\mathcal{M}}_{g,n}} \Lambda_g \psi_1^{k_1}\dots \psi_n^{k_n}}
\prod_{i=1}^n\sum_{\mu_i=1}^{\infty}
\dfrac{\mu_i^{\mu_i+k_i+1}}{\mu_i!}x_i^{\mu_i-1}dx_i\\ \nonumber
&=\sum_{\mu_1,\dots,\mu_n}\br{\int_{\overline{\mathcal{M}}_{g,n}}
\dfrac{\Lambda_g}{\prod_{i=1}^{n}(1-\mu_i\psi_i)}}
\prod_{i=1}^{n}\dfrac{\mu_i^{\mu_i+1}}{\mu_i!}x_i^{\mu_i-1}dx_i .
\end{align}
Here we used the fact that in our case
\begin{align*}
d\xi_k &= (-1)^{k+1}\Do^{k+1}d\ly = d\br{\br{x\dfrac{d}{dx}}^{k+1} y}\\ &= d\br{\br{x\dfrac{d}{dx}}^{k+1} \sum_{\mu=1}^{\infty}\dfrac{\mu^{\mu-1}}{\mu!}x^{\mu}} = 
\sum_{\mu=1}^{\infty}\dfrac{\mu^{\mu+k+1}}{\mu!}x^{\mu-1}dx .
\end{align*}

Note that the Bouchard-Mari\~no conjecture may be written as
\begin{equation}
\label{bmdifs}
w_{g,n} = \sum_{\mu_1,\dots,\mu_n} \dfrac{\ch_{g;\mu_1,\dots,\mu_n}}{\bn(g,\mu)!}
     \;\mu_1\dots\mu_n\,x_1^{\mu_1-1}\dots x_n^{\mu_n-1}dx_1\dots dx_n,
\end{equation}
while the ELSV formula states that
\begin{equation}
\ch_{g;\mu_1,\dots,\mu_n}=\bn(g,\mu)!\br{\int_{\overline{\mathcal{M}}_{g,n}}
\dfrac{\Lambda_g}{\prod_{i=1}^{n}(1-\mu_i\psi_i)}}
\prod_{i=1}^{n}\dfrac{\mu_i^{\mu_i}}{\mu_i!} .
\end{equation}

We immediately see that formula \eqref{equivf} directly implies the following
\begin{theorem} The Bouchard-Mari\~no conjecture and the ELSV formula are equivalent.
\end{theorem}

This means that we have a new proof of the ELSV formula, since we proved the Bouchard-Mari\~no conjecture independently in Section \ref{secBM}. Note that the Bouchard-Mari\~no conjecture as given in Theorem \ref{th1} is equivalent to formula \eqref{bmdifs}, if one takes into account the topological recursion formula for $w_{g,n}$, given by Equations~\eqref{eqB1} and~\eqref{eqEO}.

\end{document}